\documentclass[a4paper,twoside]{article}

  \usepackage{amsthm}
  \usepackage{amsmath}
  \usepackage{amsfonts}
  \usepackage{amssymb}
  \usepackage{amscd}
  \usepackage[mathscr]{eucal}

  \usepackage{enumitem}

  \usepackage{fancyhdr}
  \usepackage[linkcolor=black,colorlinks=true]{hyperref}

  \usepackage[top=2.5cm, bottom=2.7cm,left=3.5cm, right=2.5cm]{geometry}

  \theoremstyle{definition}  
   \newtheorem{defn}{Definition}[section]

  \theoremstyle{plain}  
   \newtheorem{thm}[defn]{Theorem}
   
   \newtheorem{prop}[defn]{Proposition}
   \newtheorem{cor}[defn]{Corollary}

  \theoremstyle{remark} 

  \numberwithin{equation}{section}
  \allowdisplaybreaks[1] 
  \setlist[enumerate]{font=\upshape} 

  \renewcommand{\bf}[1]{\textbf{#1}}
  
  \renewcommand{\sc}[1]{\textsc{#1}}

  \newcommand{\mcl}[1]{\mathcal{#1}}

\begin{document}
  \fontsize{12}{14}
  \selectfont
  \title{\LARGE Nilpotent Completely Positive Maps}
  \author{B.V. Rajarama Bhat and Nirupama Mallick}

\maketitle

\date{ }

 \begin{center}
\bf {Abstract}


\end{center}
We study the structure of nilpotent completely positive maps in
terms of Choi-Kraus coefficients. We prove several inequalities,
including certain majorization type inequalities for dimensions of
kernels of powers of nilpotent completely positive maps.



\section{Introduction}

Completely positive (CP) maps on $C^*$-algebras (\cite{wfs}) are
well-studied objects due to their importance in various contexts. In
operator algebra theory they are a major tool to understand
$C^*$-algebras. For instance nuclearity for $C^*$-algebras can be
defined using CP maps( see \cite{BO, Paulsen} ). In quantum
probability they replace Markov maps of classical probability
(\cite{Par} ). In quantum information theory they are quantum
channels(\cite{Hol}). However, in most of these studies usually one
looks at unital and some times trace preserving completely positive
maps. In the present article we initiate a study of nilpotent CP
maps. The reasons for this are two fold. Firstly, in basic linear
algebra nilpotent linear maps form an important class. They are
quite distinct from well behaved hermitian or normal maps for which
we have the powerful spectral theorem. However Jordan decomposition
gives us a good enough tool to study nilpotent maps. It is important
to study them as they appear in representation theory, Lie algebra
theory and in various other contexts (\cite{Ful}).  Secondly we came
across nilpotent CP maps when we were looking at some special
classes of unital CP maps called roots of states (\cite{Bh}). In
other words an analysis of  nilpotent CP maps is beneficial to
understand some unital CP maps. This connection would be discussed
in the last section of this article.

To begin with let us  recall some basic theory of nilpotent linear
maps on finite dimensional vector spaces. Suppose $V$ is a finite
dimensional complex vector space of dimension $n$ and suppose
$L:V\to V$ is a  nilpotent linear map. Then $p\geq 1$ is called the
order of nilpotency of $L$ if $L^p=0$ and $L^{p-1}\neq 0.$  Suppose
$L$ is nilpotent of order $p.$ Take $V_i=\mbox{ker} (L^i)$ for
$1\leq i\leq p.$ Clearly $\{0\}\subset V_1\subset V_2 \cdots \subset
V_p=V$. Suppose $l_i=\mbox{dim} V_i/V_{i-1} $, $1\leq i\leq p$ (with
$V_0=\{0\}$). Then $l_i\geq 1$ for every $i$ and $l_1+l_2+\cdots
+l_p=n$.  From the observation that $L$ induces an injective map
from $V_i/V_{i-1}$ to $V_{i-1}/V_{i-2}$, we get,  $l_1\geq l_2\geq
\cdots \geq l_p.$  In other words, $(l_1, l_2, \ldots , l_p)$ is a
partition of the natural number $n$. We call $(l_1, l_2, \ldots ,
l_p)$ as the {\em nilpotent type} of $L$. It is an important
invariant for $L$. Another way of arriving at the nilpotent type is
as follows. Suppose $(m_1, m_2, \ldots , m_k)$ are sizes of the
Jordan blocks of $L$ arranged in decreasing order. Then $(l_1, l_2,
\ldots , l_p)$ is just the dual of $(m_1, m_2, \ldots , m_k)$ in the
sense of partitions, that is,: $l_i=\#\{ j: m_j\geq i \}.$

Suppose now $M\subset V$ is an invariant subspace of $L$ and $N=V/M
.$  Take $R=L|_M$ and let $S$ be the map induced by $L$ on
quotienting in $N=V/M$. It is easy to see that $R,S$ are nilpotent
of order at most $p$. Suppose $R, S$ are of nilpotent type $(r_1,
r_2, \ldots r_p),$ $(s_1, s_2, \ldots , s_p)$ respectively (Here,
the tuples are extended by 0's if the nilpotent orders are less than
$p$). Then there are interesting inequalities connecting the three
triples, coming from the so called Littlewood-Richardson rules (see
\cite{Ap,  BLS, CF, LM}  ). For instance, we get `majorization'
inequalities: $l_1+l_2+ \ldots +l_k\leq (r_1+r_2+\cdots +r_k)+
(s_1+s_2+\cdots +s_k)$ for all $1\leq k\leq p$ and $l_1+l_2+\cdots
+l_p= (r_1+r_2+\cdots +r_p)+(s_1+s_2+ \cdots +s_p).$ The main goal
of this article is to extend some of these results to nilpotent
completely positive maps.

Now suppose $H$ is a finite dimensional Hilbert space with
dim$(H)=n$ and let $\mathcal {B}(H)$ be the algebra of all linear
maps on $H$. A linear map $\alpha : \mathcal {B}(H)\to \mathcal
{B}(H)$ is said to be completely positive if $\sum Y_i^*\alpha
(X_i^*X_j)Y_j \geq 0$ for all choices of $X_1, X_2, \ldots , X_k,
Y_1, Y_2, \ldots , Y_k$ in $\mathcal {B}(H)$ with $k\geq 1.$ Suppose
$\alpha $ is nilpotent. As dim$(\mathcal {B}(H)= n^2)$ we see
immediately that the order of nilpotency of $\alpha $ is at most
$n^2$. But actually as we see in Corollary 2.5, this order can't be
bigger than $n$. Suppose the order of nilpotency of $\alpha $ is
$p.$ In the next section we assign a tuple $(a_1, a_2, \ldots ,
a_p)$ called the CP nilpotent type of $\alpha $. Here $a_1, a_2,
\ldots , a_p$ are natural number such that $a_1+a_2+\cdots +a_p =n$,
but unlike the previous case, $a_j$'s need not be in the decreasing
order. Nevertheless in Section 3, we do obtain a majorization result
for CP nilpotent types similar to the one stated above. It is not
yet clear as to what other inequalities hold.

In the last section we recall the definition of roots of states
which are special kinds of unital completely positive maps,
appearing in dilation theory and quantum stochastics. We see as to
how they always contain a nilpotent CP map as a substructure. A
result in the converse direction is also possible. Though some of
the theory can be extended easily to the infinite dimensional case,
we do not do it here, as the main focus of this article is on some
inequalities coming from dimension counting. For this reason we
avoid all technicalities and restrict ourselves entirely to the
finite dimensional setting.

\section{Basics of nilpotent CP maps}

Let $H$ be a finite dimensional Hilbert space. Let $\alpha :\mcl
{B}(H)\to \mcl {B}(H)$ be a  completely positive map. Then it has a
Choi-Kraus (\cite{Ch, Kr} ) decomposition:
\begin{equation}
\alpha (X)= \sum _{i=1}^dL_i^*XL_i ~~\forall X\in \mcl{B}(H)
\end{equation}
for some linear maps $L_1, L_2, \ldots , L_d$ in $\mcl{B}(H)$ for
some $d\geq    1.$ The decomposition is not unique. However, if
\begin{equation*}
\alpha (X)= \sum _{j=1}^{d'}M_j^*XM_j~~ \forall X\in \mcl{B}(H)
\end{equation*}
%
then $\mcl{L}_{\alpha } := \mbox{span} \{ L_i : 1\leq i\leq d\}
=\mbox{span}\{ M_j: 1\leq j\leq d' \} .$ Following Arveson we call
$\mcl{L}_{\alpha } $ as the metric operator space of $\alpha .$ The
linear maps $L_i$ in $(2.1)$ are called Choi-Kraus coefficients.  It
is possible to choose them to be linearly independent. The number of
linearly independent Choi-Kraus coefficients '$d$' (which is also
the dimension of $\mcl{L}_{\alpha }$) is called the index of $\alpha
.$

Now onwards $\alpha :\mcl{B}(H)\to \mcl{B}(H)$ is a non-zero
completely positive map with Choi-Kraus decomposition as in (2.1).

 \begin{prop}
  Let $\alpha : \mcl{B}(H)\longrightarrow \mcl{B}(H)$ be a non-zero completely positive map with a Choi-Kraus
  decomposition $\alpha(X) = \sum_{i=1}^d L_i^*XL_i \quad\forall \;X \in \mcl{B}(H)$.  Then

        (i) $\mbox{ker}(\alpha(1))=\cap_{i=1}^d \mbox{ker}({L_i})$

       (ii) For $X \geq 0,\alpha(X)=0$  iff $\mbox{ran}(X)\subseteq \cap_{i=1}^d \mbox{ker}(L_i^*) $

       (iii) $\{ x\in H :\alpha(|x\rangle \langle x|) =0\} =\cap_{i=1}^d \mbox{ker}({L_i^*})$

        (iv)  $\mbox{ran}(\alpha(1))=\overline{span}\{L_i^*u: u\in H, 1\leq i \leq d\}$
 \end{prop}

\begin{proof} (i) We have,
\begin{align*}
   \mbox{ker}(\alpha(1))& =\{x\in H:\alpha(1)x=0\}\\
                                    &=\{x\in H:\sum_{i=1}^d L_i^*L_ix=0 \}\\
                                    &=\{x\in H:\sum_{i=1}^d \langle x, (L_i^*L_i)x\rangle =0\}\\
                                    &=\cap_{i=1}^d \mbox{ker}(L_i).
            \end{align*}

(ii)  Assume that $\alpha(X)=0$ with $X\geq 0$. Then $\sum_{i=1}^d
L_i^*XL_i=0$ and then by positivity of $X$,  $L_i^* XL_i=0$ for all
$i.$ Therefore $L_i^*X^{\frac{1}{2}}(X^{\frac{1}{2}}L_i)=0$, which
implies $L^*X^{\frac{1}{2}}=0$ and $L^*X=0$. So we get
$\mbox{ran}(X)\subseteq \bigcap _{i=1}^d\mbox{ker}(L_i^*).$

     Conversely assume that \mbox{ran}$(X)\subseteq\cap_{i=1}^d\mbox{ker}({L_i^*})$. Then $L_i^* X=0$ for all $1\leq i \leq d$. Hence  $\alpha(X)=\sum_{i=1}^d L_i^*XL_i =0$.

(iii) Since for $x\in H$, ran$(|x\rangle \langle x|)= \mathbb{C}x$
this result is clear from part (ii).


(iv)     Let $ x\in \mbox{ran} (\alpha(1))$. Then
$x=\sum_{i=1}^dL_i^*L_iy=\sum_{i=1}^dL_i^*(L_iy)$ for some $y\in H$.
So $x\in\overline{\textrm{span}}\{L_i^*u: u\in H, 1\leq i\leq d\}$.
Thus $\mbox{ran} (\alpha(1))\subseteq
\overline{\textrm{span}}\{L_i^*u: u\in H, 1\leq i\leq d\} $. For
reverse inclusion let $ x\in\overline{\textrm{span}}\{L_i^*u:u\in H,
1 \leq i\leq d\}$. Then $ x=\sum_{i=1}^d \lambda_iL_i^*u_i $ where
$u_i \in H, \lambda_i\in\mathbb{C}$. To show $ x\in \textrm{ran
}\alpha(1)=\textrm{ran}(\sum L_i^*L_i)$. It is enough to show  $x\in
\textrm{ker}(\sum L_i^*L_i)^\perp$ i.e., $\langle x,y\rangle=0$ for
all $y\in \textrm{ker}(\sum L_i^*L_i)$.  Suppose $y\in
\textrm{ker}(\sum L_i^*L_i)$.  Now $0= \langle y, \sum
_iL_i^*L_iy\rangle= \sum _i \langle L_iy, L_iy\rangle $ implies that
$L_iy=0$ for all $i$. So
     \begin{equation*}
      \langle y,x\rangle= \langle y,\sum _i\lambda_iL_i^*u_i \rangle
                        =  \sum _i\lambda_i\langle y, L_i^*u_i \rangle
                        =\sum _i\lambda_i\langle L_iy, u_i \rangle
                        =0.
    \end{equation*}
 \end{proof}

We remark that (i), (iv) of last Proposition could have been stated
as $\mbox{ker} (\alpha (1))= \bigcap \{ \mbox{ker}L : L\in \mathcal
{L}_{\alpha }$ and $\mbox{ran} (\alpha (1))=\overline
{\mbox{span}}\{ L^*u: L\in \mathcal{L} _{\alpha }, u\in H\}.$ This
way, we can make statements free of the choice of the  Choi-Kraus
representation. However we do not opt to do so.

Recall that, a linear map $A $ is said to be nilpotent of order $p$
if $A^p=0$ and $p$ is the smallest such number. We wish to look at
Choi-Kraus decompositions of nilpotent completely positive maps.
\begin{prop}
  Let $\alpha : \mcl{B}(H)\to \mcl{B}(H)$ be a non-zero completely positive map with a Choi-Kraus
  decomposition $\alpha(X) = \sum_{i=1}^d L_i^*XL_i \quad\forall \;X \in
  \mcl{B}(H).$ Then the following are equivalent.

  (i) $\alpha ^p(X)= 0 ~~~\forall X\in \mcl{B}(H);$

(ii) $L_{i_1}L_{i_2}\ldots L_{i_p}=0$ for all $i_1, i_2, \ldots
,i_p.$
\end{prop}

\begin{proof}
Clearly,
  \begin {equation*}
       \alpha^p(X)=\sum_{i_1,i_2\dots,i_p=1}^d L_{i_p}^*L_{i_{p-1}}^*\dots L_{i_1}^*XL_{i_1}L_{i_2}\dots L_{i_p}=0\quad\forall\;X\in\mcl{B}(H)
  \end{equation*}
Now the result follows easily from the previous Proposition, by
considering the CP map $\alpha ^p$ in place of $\alpha $.
\end{proof}

  Let $\alpha : \mcl{B}(H)\to \mcl{B}(H)$ be as above
  and suppose that $\alpha $ is nilpotent of order $p.$ Take
  $H_1:=\textrm{ker}(\alpha(1))$ and for $2\leq k\leq p,$
\begin{equation}\label{Hk}
 H_k:= \textrm{ker}(\alpha ^k(1))\bigcap (\textrm{ker}(\alpha
^{k-1}(1)))^{\perp}.
\end{equation}
Then clearly, $H=H_1\oplus H_2 \oplus \cdots \oplus H_p.$ Observe
that, $L_k(H_{i+1})\subseteq H_1\oplus H_2\oplus \cdots \oplus H_i$
for all $i$. In other words, Choi-Kraus coefficients have a block
strictly upper triangular decomposition. \begin{defn} Let
$a_i=\textrm{dim}(H_i)$ for $1\leq i\leq p.$ Then  $(a_1, a_2,
\ldots, a_p)$ is called the {\em CP nilpotent type} of $\alpha .$
\end{defn}

Note that $a_1+a_2+ \cdots +a_p=$ dim $H.$ From the usual theory of
nilpotent linear maps one may expect $a_{i+1}\leq a_{i}$ for all
$i$. However, it is easily seen that this is not true. The
appropriate statement is the following.

\begin{thm} \label{basicinequality} {\bf (Basic inequalities)}
With notation as above, \begin{equation} a_{i+1}\leq d.a_i ~~\forall
1\leq i\leq (p-1).\end{equation}
 Conversely, given a tuple $(a_1, a_2, \ldots , a_p)$ of natural numbers, adding up to dim $(H)$, and satisfying
 (2.3), there exists a nilpotent completely positive map $\alpha
 :\mcl{B}(H)\to \mcl{B}(H)$ of order $p$ and CP type $(a_1, a_2, \ldots
 , a_p).$
\end{thm}

\begin{proof}
 Since $\alpha$ is nilpotent of order $p$ we have $H=H_1\oplus H_2 \oplus \cdots\oplus H_p$, where $ H_k$ is given by $\eqref{Hk}$. Define $L:H\rightarrow H\otimes\mathbb{C}^d $ by
 $L(x)=\begin{pmatrix}
           L_1(x),L_2(x),\dots,L_d(x)
        \end{pmatrix}^t$. Where $'t'$ denotes the transpose of matrix.
 Note that $ L(H_{i+1})\subseteq H_{1}\oplus H_{2}\oplus\dots\oplus H_{i})\otimes\mathbb{C}^d$ for all
 $i=1,2,\dots, (p-1)$. Let  $ P_{i}$ be the  projection from $H$
 onto $H_{i}$. If we show that
  $(P_{i}\otimes I)\circ L:H_{i+1}\rightarrow H_{i}\otimes\mathbb{C}^d$ is one to one, then we have $ a_{i+1}\leq d.a_{i}$. Suppose $((P_{i}\otimes I)\circ L)(x)=0$ for $x\in H_{i+1}$.
  Then $L_k(x)\in\ker(P_{i})=H_{1}\oplus H_{2}\oplus\dots\oplus H_{i-1}=\cap\ker(L_{j_1}L_{j_2}\dots L_{j_{i-1}})$ for all $ k=1,2,\dots,d$. So
  $x\in \cap\ker(L_{j_1}L_{j_2}\dots L_{j_{i-1}}L_{j_i}) =H_{1}\oplus H_{2}\dots\oplus H_{i}$. As $x\in (H_{1}\oplus
  H_{2}\dots\oplus  H_{i})\bigcap H_{i+1}$, $x=0.$

For the converse part, arbitrarily decompose $H$ as $H=\oplus
_{i=1}^pH_i$ with dim $H_i=a_i$. Fix $i, 1\leq i<p.$ We have
$a_{i+1}\leq d.a_i$. Then by simple dimension counting it is easy to
construct $d$ linear maps $M_1, M_2, \ldots , M_d$ from $H_i$ to
$H_{i+1}$ such that span$\bigcup _k(M_k)(H_i)= H_{i+1}.$ Take
$L_k|_{H_{i+1}}=M_k^*, 1\leq k\leq d.$ Then it is clear that
$\bigcap _k\mbox{ker}(L_k)|_{H_{i+1}}=\{0\}.$ It follows that
$\alpha $ defined by $\alpha (X) =\sum _kL_k^*XL_k$ on $\mathcal
{B}(H)$ has required properties.

\end{proof}

\begin{cor}
Let $\alpha : \mcl{B}(H)\to \mcl{B}(H)$ be a completely positive map
which is nilpotent of order $p.$ Then $p\leq $ dim $(H).$
\end{cor}

\begin{proof}
If for some $i$, $a_i=0$, then from the inequality (2.3), $a_j=0$
for all $j\geq i.$ Then dim $(H)= a_1+a_2+\cdots +a_p$, clearly
yields the result.
\end{proof}

For the CP map $\alpha (X) =\sum L_i^*XL_i$ we define the conjugate
map $\alpha ^*: \mcl{B}(H)\to \mcl{B}(H)$ by $\alpha ^*(\cdot )=
\sum _iL_i(\cdot )L_i^*.$ It is not hard to see that:
$$ \mbox{trace} [\alpha (X)^*Y]= \mbox{trace} [X^*\alpha ^*(Y)]$$
for all $X, Y$ in $\mcl{B}(H)$. This justifies the notation $*$ and
also shows that the conjugate is independent of the choice of the
Choi-Kraus decomposition.

\begin{prop}
Suppose $\alpha , \alpha ^*$  are as above. Then $\alpha $ is
nilpotent of order $p$ if and only if $\alpha ^*$ is nilpotent of
order $p.$
\end{prop}
\begin{proof}
This is clear from Proposition 2.2.
\end{proof}

Nilpotent type of a nilpotent linear map $L$ and that of its adjoint
$L^*$ are same. This is no longer the case for CP nilpotent types.
If we analyze this further, we get the following picture.

Once again let $\alpha $ be a nilpotent CP map of order $p$ on
$\mcl{B}(H).$ Decompose the Hilbert space $H$ as $H= H_1\oplus
H_2\oplus \cdots \oplus H_p$ as before and also consider
$$H=H^1\oplus H^2\oplus \cdots \oplus H^p$$
where
$$H^1=\mbox{ker}(\alpha ^*(1)), ~~
 H^k= \textrm{ker}((\alpha ^*) ^k(1))\bigcap (\textrm{ker}((\alpha
 ^*)
^{k-1}(1)))^{\perp} ~~\forall k\geq 2.$$ Further take,
$a^k=\textrm{dim}(H^k)$ for $1\leq k\leq p.$ Then we get our first
majorization type result.

\begin{thm} Suppose $\alpha , \alpha ^*$ have CP nilpotent types
$(a_1, a_2, \ldots , a_p)$, $(a^1, a^2, \ldots , a^p)$ as above.
Then
$$a_{p-i+1}+a_{p-i+2}+\cdots +a_p \leq a^1+a^2+\cdots +a^i.$$
for $1\leq i\leq p$ and $a_1+a_2+\cdots +a_p=a^1+a^2+\cdots +a^p.$
\end{thm}
\begin{proof} The equality part is obvious.
We consider a Choi-Kraus decomposition of $\alpha $ as before. Now
$L_k(H_j)\subseteq H_1\oplus H_2\oplus \cdots \oplus H_{j-1}$ for
all $j,k$. Therefore ran$(L_k) $ is contained in $H_1\oplus
H_2\oplus \cdots \oplus H_{p-1}.$ It follows that $L_k^*(H_p)=\{
0\}$ or $H_p\subseteq H^1$. Hence $a_p\leq a^1.$ In a similar way,
$$L_{k_1 }L_{k_2}\ldots L_{k_i}(H)\subseteq H_1\oplus H_2\oplus \cdots \oplus
H_{p-i}$$ for all $k_1, k_2, \ldots , k_i. $  Consequently
$L_{k_1}^*L_{k_2}^* \ldots L_{k_i}^* (H_{p-i+1}\oplus
H_{p-i+2}\oplus \cdots \oplus H_p)=\{0\}.$  That is,
$H_{p-i+1}\oplus H_{p-i+2}\oplus \cdots \oplus H_p\subseteq
H^1\oplus H^2\oplus \cdots \oplus H^i$. In particular,
$a_{p-i+1}+a_{p-i+2}+\cdots +a_p \leq a^1+a^2+\cdots +a^i.$


\end{proof}

The structure is particularly interesting if the projections onto
$H^ î$'s and $H_j$'s commute. In such a case, take $m_{ij}=$
dim$(H^i\bigcap H_j)$ for $1\leq i,j \leq p.$ Then $a_j =\sum _i
m_{ij},$ and $a^i= \sum _jm_{ij}.$ Moreover, as $H_{p-i+1}\oplus
H_{p-i+2}\oplus \cdots \oplus H_p\subseteq H^1\oplus H^2\oplus
\cdots \oplus H^i$, $m_{ij}=0$ for $i+j>(p+1).$ It is not clear as
to whether there is any such structure when the associated
projections do not commute.

For a nilpotent CP map $\alpha (X ) =\sum L_i^*XL_i$,  let the usual
nilpotent type be $(l_1, l_2, \ldots , l_p)$. In other words, $l_i=
$ dim $V_i/V_{i-1}$, where $V_i= \mbox{ker}(\alpha ^i)$ for $1\leq i
\leq p$ and $V_0=\{ 0\}.$ Then $l_1\geq l_2\geq \cdots \geq l_p) $
with $l_1+l_2+\cdots +l_p= (\mbox{dim}H)^2.$
 Fix $1\leq k\leq p.$ From Proposition
2.1,  for $x$ or $y$ in the kernel of $(\alpha ^* )^k (1)$, we have
$\alpha ^k (|x\rangle \langle y|)=0$.  Hence $\sum _{i=1}^kl_i\geq
(\sum _{i=1}^k(a^i))^2+2(\sum _{i=1}^ka ^i)(\sum _{i=k+1}^na ^i),$
where  $(a^1, a^2, \ldots , a^p)$, is the CP nilpotent type of the
adjoint map $\alpha ^* .$

\section{Majorization}

 Let $\alpha : \mcl{B}(H)\to \mcl{B}(H)$ be a non-zero completely positive
 map and let $M$ be a subspace of $H$. Take $N=M^{\perp} . $ Note that
 $\mcl{B}(M)$ is a subspace of $\mcl{B}(H)$ in the natural way by
 identifying $X\in \mcl{B}(M)$ with $P_MX|_M$ ($P_M$ being the orthogonal projection onto $M).$ In block operator notation,
 $X $ is identified with:
$$\begin{bmatrix}
                      X & 0 \\
                      0 & 0
              \end{bmatrix}$$

 We say that $M$ is invariant
 under $\alpha $, if $\alpha $ leaves $\mcl{B}(M)$ invariant.
 Observe that this happens if and only if every Choi-Kraus
 coefficient leaves $N:=M^{\perp} $ invariant. Suppose,
$\alpha $ is as in $(2.1).$ Then we get
$$L_i=\begin{bmatrix}
                      B_i & 0 \\
                      D_i & C_i
              \end{bmatrix}$$
for some operators $B_i\in \mcl{B}(M), C_i\in \mcl{B}(N), D_i\in
\mcl{B}(M, N)$, $1\leq i \leq d.$ Define two new completely positive
maps $\beta $ on $\mcl{B}(M)$ and $\gamma $ on $\mcl{B}(N)$ by
restrictions of $\alpha $:
$$\beta (X) =\sum B_i^*XB_i    ~~\forall X\in \mcl{B}(M); $$
$$\gamma (X) =\sum C_i^*XC_i    ~~\forall X\in \mcl{B}(N). $$
Note that for any $i_1, i_2, \ldots, i_p$,
$$L_{i_1}L_{i_2}\ldots L_{i_p}=
\begin{bmatrix}
                      B_{i_1}B_{i_2}\ldots B_{i_p}& 0 \\
                      D_{i_1,i_2, \ldots , i_p} & C_{i_1}C_{i_2}\ldots C_{i_p}
              \end{bmatrix} $$
for some operator  $D_{i_1,i_2, \ldots , i_p}$. Now from Proposition
2.2 it follows that $\alpha$ is nilpotent of order $p$ then $\beta $
and $\gamma $ are nilpotent of order at most $p.$

As in the case of types of nilpotent linear maps on vector spaces,
one wishes to understand the relationship between CP types of
$\alpha , \beta $ and $\gamma .$ Here is the main result we have in
this direction. Let $(a_1, a_2, \ldots , a_p)$, $(b_1, b_2, \ldots ,
b_p)$, $(c_1, c_2, \ldots , c_p)$ be CP types of $\alpha , \beta $
and $\gamma $ respectively.  As usual we take $b_i=0$, when $i$ is
bigger than the order of nilpotency of $\beta $ and similarly
$c_i=0$ for $i$ bigger than the order of nilpotency of $\gamma .$

\begin{thm} (Majorization) With notation as above,

$$\sum _{i=1}^ka_i\leq \sum _{i=1}^kb_i+\sum _{i=1}^kc_i$$
for all $1\leq k\leq p$, and
$$ \sum _{i=1}^p  a_i= \sum _{i=1}^pb_i+\sum _{i=1}^pc_i.$$
\end{thm}

\begin{proof}
The equality statement follows as dim$(H)= \sum _{i=1}^pa_i$,
dim$(M)= \sum _{i=1}^pb_i$, dim$(N)= \sum _{i=1}^pc_i.$

For the inequality part, first consider the case $k=1$. Let $\{ u_1,
u_2, \ldots , u_r\}$ be a basis for $(\bigcap _i\mbox{ker}
B_i)\bigcap (\bigcap _i \mbox{ker}D_i)$. Similarly, let $\{ v_1,
v_2, \ldots , v_{c_1}\}$ be a basis for $\bigcap _i \mbox{ker}C_i.$
Then it is clear that
$$
\{ \begin{pmatrix} u_1\\0\end{pmatrix},
     \begin{pmatrix} u_2\\0\end{pmatrix},
     \ldots,
     \begin{pmatrix} u_r\\0\end{pmatrix},
     \begin{pmatrix} 0\\v_1\end{pmatrix},
     \begin{pmatrix} 0\\v_2\end{pmatrix},
     \ldots,
     \begin{pmatrix} 0\\v_{c_1}\end{pmatrix} \}$$
is linearly independent in $\bigcap _i \mbox{ker} L_i .$ Extend this
collection to:
$$
\{ \begin{pmatrix} u_1\\0\end{pmatrix},
     \begin{pmatrix} u_2\\0\end{pmatrix},
     \ldots,
     \begin{pmatrix} u_r\\0\end{pmatrix},
     \begin{pmatrix} 0\\v_1\end{pmatrix},
     \begin{pmatrix} 0\\v_2\end{pmatrix},
     \ldots,
     \begin{pmatrix} 0\\v_{c_1}\end{pmatrix},
 \begin{pmatrix} x_1\\y_1\end{pmatrix},
     \begin{pmatrix} x_2\\y_2\end{pmatrix},
     \ldots,
     \begin{pmatrix} x_s\\y_s\end{pmatrix}
     \}$$
 a basis of $\bigcap _i \mbox{ker} L_i .$ In particular, $a_1=
 r+c_1+s.$ Now we observe that $x_1, x_2, \ldots , x_s$ are vectors
 in $\bigcap _i \mbox{ker} B_i $ and we claim that $\{ u_1, u_2,
 \ldots , u_r, x_1, x_2, \ldots , x_s \}$ are linearly independent
 in  $\bigcap _i \mbox{ker} B_i .$ Once we prove this claim, we have
 $r+s\leq b_1$ and hence $a_1\leq b_1+c_1.$

 Proof of the claim is simple linear algebra as follows. Suppose
 $\sum _jp_ju_j+\sum _jq_jx_j=0$ for some scalars $p_j, q_j$.
Fix $1\leq i\leq d.$ Then as $u_j \in \bigcap _i \mbox{ker} D_i$ for
all $j$, $\sum
 _jq_jx_j \in   \mbox{ker} D_i $. Further as $\begin{pmatrix}
 x_j\\y_j\end{pmatrix}$ is in $\mbox{ker}L_i,$ $D_ix_j+C_iy_j=0$
or $C_iy_j= -D_ix_j$ for all $j$.  Consequently $\sum _j  C_iq_jy_j=
-\sum _jD_iq_jx_j=0.$ Therefore, $\sum _jq_jy_j\in \bigcap _i
\mbox{ker} C_i$, and so there exist scalars $t_i, 1\leq i\leq c_1$,
such that $\sum _it_iv_i= -\sum _jq_jy_j.$ Then, $\sum _jp_j
\begin{pmatrix} u_j\\0\end{pmatrix} + \sum _jt_j  \begin{pmatrix} 0\\v_j\end{pmatrix}
+\sum _jq_j \begin{pmatrix} x_j\\y_j\end{pmatrix} =0.$ Now due to
linear independence of these vectors, $p_j\equiv 0, q_j\equiv0$ and
$t_j\equiv 0 .$ In a similar way, by considering $\alpha ^k(1),
\beta ^k (1) $ and $\gamma ^k(1)$, $\sum _{j=1}^k a_j \leq \sum
_{j=1}^k b_j +\sum _{j=1}^kc_j $ for $1\leq k\leq p$ as $\sum
_{j=1}^ka_j= $dim $\mbox{ker} (\alpha ^k(1))$.

\end{proof}

It is to be noted that majorization here is quite different from the
standard  majorization theory (See \cite{Bhatia}), where given two
vectors, $x, y$ in $\mathbb{ R} ^n$, with $\sum _ix_i=\sum _iy_i$,
one says $y $ is majorized by $x$, if $\sum _{i=1}^k y^{\downarrow}
_i\leq \sum _{i=1}^kx^{\downarrow} _i$, for $1\leq k\leq n$, where
$`\downarrow '$ indicates that vectors are arranged in decreasing
order before comparing. In the results here, no rearrangement of
vectors are being done. In this context, it is interesting to
identify the extreme points of the convex, compact set of
non-negative vectors majorized (without ordering) by the given
vector. We get the following result, which we state without proof.
The proof is straightforward.

Let $x=(x_1, x_2, \ldots , x_n)$ be a vector in $\mathbb {R}^n$ with
$x_i>0$ for all $i.$ Consider $C(x):= $ $$\{ y: y\in \mathbb{R}^n,
y_i\geq 0 ~~\forall i, \sum _{i=1}^ky_i\leq \sum _{i=1}^kx_i
~~\mbox{for}~~~1\leq k\leq n, \sum _{i=1}^ny_i=\sum _{i=1}^nx_i\} .
$$ Then the set of extreme points of $C(x)$ is given by
$E(x):= $
 $$\{ y: y\in \mathbb{R}^n,
y_i\geq 0 ~~ \forall i, y_i= 0~\mbox{ or}~ \sum _{i=1}^ky_i= \sum
_{i=1}^kx_i ~~\mbox{for}~~~1\leq k\leq (n-1), \sum _{i=1}^ny_i=\sum
_{i=1}^nx_i\} .
$$
In particular, $E(x)$ has $2^{(n-1)}$ extreme points.

\section{Roots of states}

 We consider the following notion based on \cite{Bh}.

\begin{defn} Let $H$ be a finite dimensional Hilbert space and
let $u\in H$ be a unit vector in $H$. Consider the pure state
$X\mapsto \langle u, Xu\rangle I$ (our inner products are
anti-linear in the first variable) on $\mathcal {B}(H).$ Then a
unital completely positive map $\tau : \mathcal {B}(H)\to \mathcal
{B}(H)$ is said to be an $n$th root of this state if $$\tau ^n(X)=
\langle u , Xu\rangle I ~~~\forall X\in \mathcal {B}(H).$$
\end{defn}

Various examples of roots of states and also continuous versions of
this and its relevance to dilation theory can be found in \cite{Bh}.
Here we restrict ourselves connecting the notion to nilpotent CP
maps in the discrete finite dimensional situation. The idea is
simple and should be clear from the following Theorem.

 \begin{thm}
  Let $\tau: \mcl{B}(H)\to \mcl{B}(H)$ be a unital CP- map such that $\tau^p(X)=\langle u,Xu\rangle I$ where u is a unit vector of H.
  Set $H_0 =\{x\in H:\langle x,u\rangle=0\}$ so that $H=\mathbb{C}u \oplus H_0$.
  Suppose $\alpha:\mcl{B}(H_0)\to\mcl{B}(H_0)$ is the compression of $\tau $ to $B(H_0)$, then
$\alpha $ is nilpotent CP map of order at most $p$.
\end{thm}

 \begin{proof} For $m\geq p$, $\tau ^m(X)= \tau ^{m-p}(\tau ^p(X))= \tau ^{m-p}(\langle u , Xu\rangle I)
= \langle u , Xu\rangle \tau ^{m-p} (I)= \langle u, Xu\rangle I.$
Now for $0\leq m \leq p ,$
\begin{eqnarray*} 1&=& \langle u, \tau ^{m+p}(|u\rangle \langle u|)u\rangle
= \langle u, \tau  ^{p}(\tau ^m(|u\rangle \langle u|))u\rangle \\
&=& \langle u, \{ \langle u, \tau ^m(|u\rangle \langle u|)u\rangle
I\}u\rangle =\langle u, \tau ^m(|u\rangle \langle u|)u\rangle .
\end{eqnarray*}
Hence $\langle u, \tau ^m( |u\rangle \langle u|)u\rangle=1$ for all
$m$. As $\tau $ is contractive, it follows that $\tau ^m (|u\rangle
\langle u|)\geq |u\rangle \langle u|$ for all $m.$ Consequently
$\tau ^m(I-|u\rangle \langle u|)\leq (I-|u\rangle \langle u|)$ for
all $m$. Now complete positivity of $\tau $ yields,
  $\tau ^m \begin{pmatrix}
             0&0\\
             0&X\\
      \end{pmatrix}=
      \begin{pmatrix}
             0&0\\
      0&\alpha ^m(X)\\
  \end{pmatrix}$ for all $X\in B(H_0).$ But then,  $\alpha^p(X)=0$ for all
  $X\in\mcl{B}(H_0)$,
  as  $$               \begin{pmatrix}
                     0&0\\
                     0&\alpha ^p(X)
              \end{pmatrix} =\tau ^p\begin{pmatrix}
                    0&0\\
                    0&X
              \end{pmatrix}
 =\left\langle \begin{pmatrix}
                                            u\\
                                            0
                                     \end{pmatrix},
                                     \begin{pmatrix}
                                            0&0\\
                                            0&X
                                     \end{pmatrix}
                                     \begin{pmatrix}
                                           u\\
                                           0
                                     \end{pmatrix}
                          \right\rangle I =0.$$

 \end{proof}

We have the following result in the converse direction. Given a
nilpotent CP map we construct a root of state in a slightly larger
space.

 \begin{thm}
  Let $\alpha:\mcl{B}(H_0)\to\mcl{B}(H_0)$ be a  contractive CP-map such that $\alpha^p(Y)=0$ for all $Y\in\mcl{B}(H_0)$. Take $H=\mathbb{C}\oplus H_0$ and
  $u=\begin{pmatrix}
            1\\
            0
  \end{pmatrix}$.
  Suppose $\tau:\mcl{B}(H)\to\mcl{B}(H)$ is a map defined by
  \begin{align*}
       \tau(\begin{pmatrix}
                  X_{11}&X_{12}\\
                  X_{21}&X_{22}
                 \end{pmatrix})
            =\begin{pmatrix}
                    X_{11}&0\\
                   0&\alpha(X_{22})+X_{11}(I-\alpha(I))
             \end{pmatrix}\\
  \end{align*}
  for all $X=[X_{ij}]\in\mcl{B}(\mathbb{C}\oplus H_0)$. Then $\tau$ is a CP-map and $\tau^p(X)=\langle u,Xu\rangle I$ for all $X\in\mcl{B}(H)$.
 \end{thm}

 \begin{proof} Clearly as $\alpha $ is contractive, $I-\alpha (I)$
 is positive and hence $\tau $ is completely positive.  For $X=[X_{ij}]\in\mcl{B}(\mathbb{C}\oplus H_0)$,
  through mathematical induction, it is easily seen  that
  \begin{align*}
      \tau ^k(X)&=\begin{pmatrix}
                      X_{11}&0\\
                      0&\alpha ^k(X_{22})+X_{11}(I-\alpha ^k(I))
              \end{pmatrix}.
  \end{align*}
  In particular, as $\alpha ^p(X_{22})= \alpha ^p (I)=0$, $\tau
  ^p(X)= X _{11}I= \langle u, Xu\rangle I.$

 \end{proof}

\noindent {\bf Acknowledgements:} The second author thanks the
National Board for Higher Mathematics (NBHM), India for financial
support.

\noindent {\sc Statistics and Mathematics Unit,\\
 Indian Statistical Institute,\\
R V
College Post, Bangalore 560059, India. } \\
\texttt{bhat@isibang.ac.in } and \texttt{ nirupama@isibang.ac.in } \\

\end{document}